\documentclass[preprint,12pt]{article}
\usepackage{hyperref}
\usepackage[T1]{fontenc}
\usepackage[utf8]{inputenc}
\usepackage{graphicx}
\usepackage{amsmath,amssymb,amsfonts}
\usepackage{amsthm}
\usepackage[left=2.5cm,right=2.5cm,top=2.5cm,bottom=2.5cm]{geometry}
\usepackage{url}
\usepackage{cite}

\theoremstyle{plain}
\newtheorem{theorem*}{Theorem}
\newtheorem{conjecture*}{Conjecture}
\newtheorem{theorem}{Theorem}[section]
\newtheorem{lemma}[theorem]{Lemma}
\newtheorem{proposition}[theorem]{Proposition}
\newtheorem{corollary}[theorem]{Corollary}

\theoremstyle{definition}
\newtheorem{definition}[theorem]{Definition}
\newtheorem{conjecture}[theorem]{Conjecture}

\newtheorem{manualtheoreminner}{Theorem}
\newenvironment{manualtheorem}[1]{%
  \renewcommand\themanualtheoreminner{#1}%
  \manualtheoreminner
}{\endmanualtheoreminner}

\newtheorem{manualcorinner}{Corollary}
\newenvironment{manualcor}[1]{%
  \renewcommand\themanualcorinner{#1}%
  \manualcorinner
}{\endmanualcorinner}

\title{\bf{Upper Bounds on the Chromatic Index of Linear Hypergraphs}}

\author{Thomas Murff$^{\, 1, }$\footnote{\url{ thomasjmurff@gmail.com }}  { }  {  }   Xerxes D. Arsiwalla$^{\, 1, }$\footnote{\url{x.d.arsiwalla@gmail.com}}   \\  
{}  \\
{\it \small  $^{1}$Wolfram Institute for Computational Foundations of Science, IL, USA}   
}

\date{}

\begin{document}

\maketitle

\begin{abstract}
We address the problem of finding upper bounds on the chromatic index $q(V,E)$ of linear (and loopless) hypergraphs. The first bound we find is defined through a color-preserving group on a proper and minimally edge-colored linear hypergraph, whose orbits serve as a finer partition to the hypergraph's coloring, thereby yielding an upper bound on $q(V,E)$. The next set of theorems in this paper relates to combinatorial properties of hypergraph coloring. Our results suggest a plausible approach to solving the Berge-F\"{u}redi conjecture, providing an upper bound on the chromatic index that directly relates $q(V,E)$ and $\Delta([(V,E)]_{2}) + 1$. Furthermore, we provide three sufficient conditions for the conjecture to hold within this framework, when involving the Helly property for hypergraphs.
\end{abstract}

\vspace{2pc}

\clearpage

\section{Introduction and Main Theorems}

Vizing's Theorem states that, for an undirected graph $G$, $G$ may be edge-colored such that 
\[  q(G) \leq \Delta(G) + 1 \]  
for the \emph{chromatic index} $q(G)$, the \emph{maximum degree} $\Delta(G)$, and for a \emph{proper coloring} (no adjacent edges are colored the same). There is a generalized version of this conjecture for linear, loopless hypergraphs \cite{Berge1997}
 \cite{BrettoFaisantHennecart2025}, which are hypergraphs where $|V_{i} \cap V_{j}| \leq 1$ for any $V_{i},V_{j} \in E$, as well as for any $V_{i} \in E$ one assumes $|V_{i}| > 1$. This also uses the notion of the \emph{2-section} of a hypergraph $(V,E)$, which is the graph denoted $[(V,E)]_{2}$-where every hyperedge is replaced with a complete graph on its particular vertices.

\begin{conjecture}[Berge-F\"{u}redi conjecture]
    A linear (loopless) hypergraph $(V,E)$ gives the inequality 
    \begin{eqnarray}
        q(V,E) \leq \Delta([(V,E)]_{2}) + 1
    \end{eqnarray}
\end{conjecture}

The core theorem of the group theory side to the paper is the following upper bound, where $\boldsymbol{T}$ is a subgroup of the automorphism group of $(V,E)$ related to hyperedge coloring.

\begin{manualtheorem}{1}[Group-theoretic bound]\label{thm: Burnside}
 \begin{eqnarray}
     q(V,E) \leq \frac{1}{|\boldsymbol{T}|} \cdot \sum_{t \in \boldsymbol{T}} |f^{-1}(\boldsymbol{1})^{t}| 
 \end{eqnarray}
\end{manualtheorem}

\vspace{3mm}

The main results for the combinatorics side of the paper are as follows, relating the relevant terms of the conjecture with inequalities, along with the new terms $|\Gamma_{c_{0}}|$ and $H_{\Gamma}$.

\begin{manualtheorem}{2}[Combinatorial bound]
    \begin{eqnarray}
        q(V,E) + |\Gamma_{c_{0}}|\leq \frac{|\Gamma_{c_{0}}|+1}{ar(V,E) - 1}\cdot \Delta([(V,E)]_{2}) + 1
    \end{eqnarray}  
    where $\Gamma_{c_{0}}$ varies over the colors $c_{0} \notin \bar{C}^{*}([v]_{\sim})$, for a fixed $v \in V$ with two-section maximum degree, i.e. $deg_{2}(v) = \Delta([(V,E)]_{2})$.
\end{manualtheorem}

Lastly, we give sufficient conditions for the Berge-F\"{u}redi conjecture to hold, involving the "Helly-ness" of a hypergraph $H_{\Gamma}$ derived from our initial hypergraph.

\begin{manualtheorem}{2.1}
    Let $(V,E)$ be a linear loopless hypergraph and $\bar{C}$ a proper minimal coloring. Also, let $ar(V,E) \geq 3$. If $(V,E)$ and $\bar{C}$ possess a derived hypergraph $H_{\Gamma}$ that has the Helly property, then $q(V,E) \leq \Delta([(V,E)]_{2}) + 1$.
\end{manualtheorem}

\begin{manualcor}{2.2}\label{cor:gamma-tau}
    Let $(V,E)$ be a $k$-uniform, linear hypergraph for $k \geq 3$, with a proper minimal hyperedge coloring $\bar{C}$. Suppose further that, for fixed $v \in V$ with $deg_{2}(v) = \Delta([(V,E)]_{2})$, there exists a clique of hyperedges (in that it is pairwise-intersecting) $\mathcal{F} \subseteq E$ with the nonempty set of colors $C/\bar{C}^{*}([v]_{\sim})$, such that $\bar{C}(\mathcal{F}) = C/\bar{C}^{*}([v]_{\sim})$. Then, $|\mathcal{F}| > k^{2} -k + 1$ implies $q(V,E) \leq \Delta([(V,E)]_{2}) + 1$.
\end{manualcor}

\section{Notation and Preliminaries}
To begin, there is some notation and background material to establish:
\begin{itemize}
    \item $q(V,E)$ is the chromatic index of the hypergraph $(V,E)$, which is the minimal number of colors such that a proper hyperedge coloring exists (often denoted with $\chi^{'}$).
    \item $deg_{2}(u)= \sum_{V_{i} \in [\{u\}, V] \cap E} (|V_{i}| - 1)$.
    \item $deg(u) = |[\{u\}, V] \cap E|$.
    \item $\Delta(V,E)$ is the max degree of $(V,E)$, meaning the largest number of hyperedges containing a given vertex.
    \item $\Delta([(V,E)]_{2})$ is the max degree of the two-section graph denoted $[(V,E)]_{2}$.
    \item $ar(V,E)$ and $r(V,E)$ are the antirank-the smallest size amongst the hyperedges, and the rank-the largest size amongst the hyperedges, respectively.

    \item $\bar{C}$ denotes the actual map that colors the hyperedges with colors from its underlying set $C$.
\end{itemize}

A \emph{hypergraph} can be thought of as a generalization of a \emph{graph}, in that its edge set consists of a collection of subsets of its own vertex set. In this way, a graph is a specific hypergraph whose edges are subsets of precisely two vertices each. For this paper, we will only be considering linear (every pair of hyperedges intersect on at most one vertex) hypergraphs.

\begin{definition}
    For a finite vertex set $V$, a collection $E$ of subsets of $V$ that define our edge set, and the power set $\mathcal{P}(V)$, we can formally identify a hypergraph $(V,E)$ with the inclusion mapping $i: E \hookrightarrow P(V)$.
\end{definition}

A key notion of this paper is to consider an alternative definition of a hypergraph.

\begin{definition}
    For a hypergraph $(V,E)$, define $f: \mathcal{P}(V) \rightarrow \{\boldsymbol{0},\boldsymbol{1}\}$ by $f(V_{i}) = \boldsymbol{1}$ if and only if $V_{i} \in E$, and $f(V_{i}) = \boldsymbol{0}$ otherwise.
\end{definition}

This definition of a hypergraph as being encoded by a Boolean-valued function $f$ allows one to view the various subset relationships between both hyperedges (elements of the preimage $f^{-1}(\boldsymbol{1})$) and non-hyperedge subsets (elements of $f^{-1}(\boldsymbol{0})$). Specifically, these various relationships can be framed within the context of order theory, in the form of the power set. Throughout the rest of the paper, $f^{-1}(\boldsymbol{1})$ and $E$ are used interchangeably to denote the hyperedge set. We can formalize the power set with the following definition.

\begin{definition}
    The power set lattice can be defined as $\boldsymbol{\mathcal{P}(V)}:= \langle \mathcal{P}(V), \cap, \cup, \emptyset, V, ()^{c} \rangle$, where $\cap$ and $\cup$ denote set intersection and set union respectively, $V$ is the entire set, $\emptyset$ is the empty set, and $()^{c}$ is set complementation.
\end{definition}

Note that this is not the algebraic presentation for an arbitrary finite lattice, as there are additional elements to the structure. In general, lattices like the power set are called \emph{Boolean algebras}, which possess complementation, least and greatest elements, and an additional distributivity property which is akin to set distributivity of $\cap$ and $\cup$. Any finite Boolean algebra is (lattice) \emph{isomorphic} to the power set lattice of some finite set.

\vspace{2mm}

For the purpose of this paper, the power set is most usefully viewed as a \emph{metric space}\cite{DeWinterKorb2016}, along with the standard algebraic properties we have discussed thus far. The definition of the power set as a metric space aligns with the definition of a \emph{Hamming Space} as being a metric space. A Hamming space is a metric space on the set of Boolean-valued strings of length $n$, where the Hamming metric distance $d(\bar{x},\bar{y})$ is defined as the bit difference between strings $\bar{x}$ and $\bar{y}$. The interesting point of connection here is that this collection of Boolean-valued strings defines a finite Boolean algebra as well, which must be isomorphic to the power set lattice of a finite set. This exact correspondence allows us to define an analog of the Hamming distance for the power set lattice.

\begin{definition}
    For the power set $\mathcal{P}(V)$ of a finite set $V$, define the distance function $d(V_{i}, V_{j}):= |V_{i} \hspace{1mm} \Delta \hspace{1mm} V_{j}|$, where $\Delta$ represents the symmetric difference of sets. Then, $d$ is metric on $\mathcal{P}(V)$, so $\mathcal{P}(V)$ equipped with $d$ is a metric space.
\end{definition}

One can see here that the metric $d$ is just a measure of how many differing elements there are between two subsets. This is indeed analogous to the Hamming distance, as the presence of a $1$ in the $i$th entry of an $n$-length Boolean string is akin to the presence of the $i$th element in the given subset of an indexed set with $n$ elements. Thus, each string of length $n$ corresponds exactly to a specific subset of the power set, given some assignment of each element of the set with some entry/index.

\vspace{2mm}

As a brief aside, a relationship to emphasize about this portion of the material is the connection between \emph{satisfiable Boolean functions} and hypergraphs. Any satisfiable Boolean function has at least one tuple of 0s and 1s such that the function evaluates the tuple as having value $\boldsymbol{1}$. Since the underlying domain of a Boolean function is just a Hamming space, the codomain the two element Boolean algebra, and as established above we have the correspondence between each Hamming space and a finite power set metric space, then we can define an associated hypergraph to the Boolean function, which we know will have at least one hyperedge. This relationship holds in the other direction as well, by assigning the input hypergraph its associated satisfiable Boolean function. So, properties of Boolean functions as well as those of hypergraphs can be translated across this bridge and studied in tandem, potentially offering links in computational complexity between the two classes of objects.

\vspace{3mm}

Now, we can continue on with the group theoretic aspects of this paper. The type of group we are interested in is an \emph{isometry group}. An \emph{isometry} is simply a function between metric spaces, that preserves the metrics between both spaces.

\begin{definition}
    For metric spaces $\boldsymbol{X}$ and $\boldsymbol{Y}$, an isometry $\phi: X \rightarrow Y$ is a function between the underlying sets such that, for all $x_{1}, x_{2} \in X$, $d_{X}(x_{1}, x_{2}) = d_{Y}(\phi(x_{1}), \phi(x_{2}))$.
\end{definition}

We can consider the collection of all bijective isometries from a metric space to itself and define it as an \emph{automorphism group}, with function composition $\circ$, function inverses, and the identity map, defining the nature of the group. For our purposes, given a finite set $V$, we can treat $\mathcal{P}(V)$ as a metric space and define the group of isometries on it, $\boldsymbol{Iso(\mathcal{P}(V))}$. The group-theoretic analysis of the Hamming space and its isometry group in \cite{DeWinterKorb2016} inspired this group-theoretic approach to the power set, per the correspondence from the Boolean algebra (lattice) isomorphism. This group is the basis of the group-theoretic foundations for this paper, as it is the group whose elements will be drawn from to study hypergraphs encoded in the power set.

\vspace{2mm}

Take a hypergraph with a vertex set $V = \{v_{1}, ... ,v_{n}\}$, an edge set $E \subseteq \mathcal{P}(V)$, and a Boolean-valued function $f$ encoding it. The crux of the group theory connection is in selecting particular induced isometries from $\boldsymbol{Iso(\mathcal{P}(V))}$ that respect our Boolean-valued function, in that they only maps hyperedges to other hyperedges, and non-hyperedge subsets to other non-hyperedge subsets. This can be formalized with the following definitions.

\begin{definition}
    For a vertex set $V = \{v_{1}, ... ,v_{n}\}$, consider $\boldsymbol{Sym(V)}$, the symmetric group of permutations on the vertex set. Each $\pi \in \boldsymbol{Sym(V)}$ is just a bijective function on $V$. We can define a corresponding isometry $(\pi)^{'}$ on the power set $\mathcal{P}(V)$ as a metric space, whose action is given by $(\pi)^{'}(V_{i}):= \{\pi(v_{i}), ... , \pi(v_{j})\}$, for any subset $V_{i} = \{v_{i}, ... , v_{j}\}$.
\end{definition}

These isometries induced from permutations define an explicit correspondence between the \emph{symmetric group} on the vertex set and the isometry group of the power set of the vertex set (as a metric space), specifically in the form of a subcollection of isometries characterized by only mapping subsets of a given size to other subsets of same size, due to the underlying bijections. This correspondence can be viewed as a rigorous group-theoretic embedding.

\begin{proposition}
    Define $()^{'}: \boldsymbol{Sym(V)} \rightarrow \boldsymbol{Iso(\mathcal{P}(V))}$, given by $(\pi)^{'}$ as defined above, for each $\pi \in \boldsymbol{Sym(V)}$. Then, $()^{'}$ is an injective group homomorphism.
\end{proposition}

\begin{proof}
    Take $(\pi_{1} \cdot \pi_{2})^{'}$ and $(\pi_{1})^{'} \circ (\pi_{2})^{'}$. 
    
    For any $V_{i} \in \mathcal{P}(V)$ with $V_{i} = \{v_{i}, ... ,v_{j}\}$, $(\pi_{1} \cdot \pi_{2})^{'}(V_{i}) = \{\pi_{1} \cdot \pi_{2}(v_{i}), ... , \pi_{1} \cdot \pi_{2}(v_{j})\}$, and $(\pi_{1})^{'} \circ (\pi_{2})^{'}(V_{i}) = (\pi_{1})^{'}(\{\pi_{2}(v_{i}), ... , \pi_{2}(v_{j})\}) = \{\pi_{1}(\pi_{2}(v_{i})), ... , \pi_{1}(\pi_{2}(v_{j}))\}= \{\pi_{1} \cdot \pi_{2}(v_{i}), ... , \pi_{1} \cdot \pi_{2}(v_{j})\}$, so $()^{'}$ is a group homomorphism. Furthermore, $()^{'}$ is injective. To see this, consider that if $(\pi_{1})^{'} = (\pi_{2})^{'}$ then both induced isometries have the exact same action on all subsets of vertices. Specifically, for every $v \in V$ we have $(\pi_{1})^{'}(\{v\}) = (\pi_{2})^{'}(\{v\})$, which is just $\{\pi_{1}(v)\} = \{\pi_{2}(v)\}$. Since $\pi_{1}$ and $\pi_{2}$ are bijections with the same action on every $v \in V$, we get that $\pi_{1} = \pi_{2}$.
\end{proof}

Now, since $()^{'}: \boldsymbol{Sym(V)} \hookrightarrow  \boldsymbol{Iso(\mathcal{P}(V))}$ is an injective group homomorphism, the image $\boldsymbol{(Sym(V))^{'}}$ defines a subgroup of $\boldsymbol{Iso(\mathcal{P}(V))}$. Thinking back to our encoding of a hypergraph with a Boolean-valued function $f: \mathcal{P}(V) \rightarrow \{\boldsymbol{0}, \boldsymbol{1}\}$, we need to look for elements of $\boldsymbol{(Sym(V))^{'}}$ that respect $f$. We define what it means for an isometry of $\boldsymbol{Iso(\mathcal{P}(V))}$ in general to respect such a function $f$, as well as the particular subgroup that these isometries define.

\begin{definition}
    Define $H:=\{\forall \phi \in \boldsymbol{Iso(\mathcal{P}(V))}|: f(\phi(V_{i}))= f(V_{i}) = f(\phi^{-1}(V_{i})) ;\forall V_{i} \in \mathcal{P}(V)\}$.
\end{definition}

\begin{proposition}
    $\boldsymbol{H}$ is a subgroup of $\boldsymbol{Iso(\mathcal{P}(V))}$.
\end{proposition}

\begin{proof}
    Consider $\phi \in H$. Then for any $V_{i} \in \mathcal{P}(V)$, $f(\phi(V_{i})) = f(V_{i}) = f(\phi^{-1}(V_{i}))$. Clearly, $\phi^{-1} \in H$ by the symmetric aspect of the definition of $H$ with respect to $f$, which gives closure under inverses. For $\phi_{1}, \phi_{2} \in H$, consider $\phi_{1} \circ \phi_{2}$. Since $f(V_{i}) = f(\phi_{2}(V_{i})) = f(\phi_{1} \circ \phi_{2}(V_{i}))$ and $f(V_{i}) = f(\phi_{1}^{-1}(V_{i})) = f(\phi_{2}^{-1} \circ \phi_{1}^{-1}(V_{i}))$, then $f(\phi_{1} \circ \phi_{2}(V_{i})) = f(V_{i}) = f(\phi_{2}^{-1} \circ \phi_{1}^{-1}(V_{i}))$ for all $V_{i} \in V$. So we have that $\phi_{1} \circ \phi_{2}$ and $\phi_{2}^{-1} \circ \phi_{1}^{-1}$ both belong to H and their $f$ values are the same, which gives closure under the group operation. Thus, $\boldsymbol{H}$ is a subgroup.
\end{proof}

Furthermore, we want to specify a subgroup of $\boldsymbol{H}$ that solely consists of isometries induced by permutations of $\boldsymbol{Sym(V)}$.

\begin{definition}
    $\boldsymbol{(Sym(V))^{'}} \cap \boldsymbol{H}$ is defined as the automorphism group of our hypergraph.
\end{definition}

 From now on, we focus our attention to the action of $\boldsymbol{(Sym(V))^{'}} \cap \boldsymbol{H}$ with respect to $f^{-1}(\boldsymbol{1})$.

\section{Group-Theoretic Bound on the Chromatic Index}

\hspace{6mm}To recall, proving the Berge-F\"{u}redi conjecture means proving the inequality $q(V,E) \leq \Delta([(V,E)]_{2})+1$ for all linear (loopless) hypergraphs $(V,E)$ where $q(V,E)$ denotes the \emph{chromatic index} of $(V,E)$. This is associated with a \emph{proper}-meaning that any intersecting hyperedges must have different colors-\emph{minimal} surjective hyperedge coloring map $\bar{C}: f^{-1}(\boldsymbol{1}) \rightarrow C$ for a set of colors $C$, such that $|\bar{C}(f^{-1}(\boldsymbol{1}))|=q(V,E)$. The minimality condition means that $q(V,E)$ is the smallest value for which such a $\bar{C}$ exists. $[(V,E)]_{2}$ denotes the two-section graph formed by replacing each $V_{i} \in f^{-1}(\boldsymbol{1})$ with a complete graph on the vertices of the given hyperedge.

\vspace{2mm}

Consider a linear hypergraph $(V,E)$ encoded by a Boolean-valued function $f: \mathcal{P}(V) \rightarrow \{\boldsymbol{0},\boldsymbol{1}\}$. Take its hypergraph automorphism group $\boldsymbol{(Sym(V))^{'} \cap H} \subseteq \boldsymbol{Iso(\mathcal{P}(V))}$. The crux of deriving the main bound of the paper comes from defining a particular subgroup of $\boldsymbol{(Sym(V))^{'} \cap H}$ that respects the coloring information-for our proper minimal coloring $\bar{C}$-of $(V,E)$ wrapped up in the automorphism information. 

\vspace{2mm}

Note that there can be a collection of distinct groups, each associated to a distinct proper/minimal coloring function of $(V,E)$. However, for our case, the specific goal in mind is to identify the number of colors, not precisely how the colors are assigned to the hyperedges, so any minimal proper coloring function will suffice, as well as whatever group structure that comes with it.

\begin{definition}
    Let $T:=\{\forall (\rho)^{'} \in \boldsymbol{(Sym(V))^{'} \cap H}|: \bar{C}((\rho)^{'}(V_{i})) = \bar{C}(V_{i}) = \bar{C}((\rho^{-1})^{'}(V_{i})); 
    \forall V_{i} \in f^{-1}(\boldsymbol{1})\}$.
\end{definition}

\begin{proposition}
    $\boldsymbol{T}$ is a subgroup of $\boldsymbol{(Sym(V))^{'} \cap H}$.
\end{proposition}

\begin{proof}
    This proof essentially follows the same line of reasoning as the proof of $\boldsymbol{H}$ being a subgroup. Firstly, the definition entails that $\bar{C}((\rho)^{'}(V_{i})) = \bar{C}(V_{i}) = \bar{C}((\rho^{-1})^{'}(V_{i}))$ for all $V_{i} \in f^{-1}(\boldsymbol{1})$, which automatically gives closure under inverses. Secondly, take $(\rho_{1})^{'}, (\rho_{2})^{'} \in \boldsymbol{T}$. Then for any $V_{i} \in f^{-1}(\boldsymbol{1})$, we have that $\bar{C}(V_{i}) = \bar{C}((\rho_{1})^{'}(V_{i})) = \bar{C}((\rho_{2} \cdot \rho_{1})^{'}(V_{i}))$ and $\bar{C}(V_{i}) = \bar{C}((\rho_{2}^{-1})^{'}(V_{i})) = \bar{C}((\rho_{1}^{-1} \cdot \rho_{2}^{-1})^{'}(V_{i}))$, so $\bar{C}((\rho_{2} \cdot \rho_{1})^{'}(V_{i})) = \bar{C}(V_{i}) = \bar{C}((\rho_{1}^{-1} \cdot \rho_{2}^{-1})^{'}(V_{i}))$. Thus, $(\rho_{2})^{'} \circ (\rho_{1})^{'}$ and $(\rho_{1}^{-1})^{'} \circ (\rho_{2}^{-1})^{'}$ are both in $T$ and have the same $\bar{C}$ values for all $V_{i} \in f^{-1}(\boldsymbol{1})$, giving closure of the operation. So $\boldsymbol{T}$ is a subgroup of $\boldsymbol{(Sym(V))^{'} \cap H}$.
\end{proof}

Notice that when considering the action of $\boldsymbol{T}$ on $f^{-1}(\boldsymbol{1})$, every orbit of the action must contain only elements of precisely the same color under $\bar{C}$. I.e. no two differently-colored hyperedges can exist in the same orbit of $\boldsymbol{T}$. It may be the case that two hyperedges with the same color are in different orbits. Consequently, the nature of the orbits as a partition of $f^{-1}(\boldsymbol{1})$ means that every color in $C$ is assigned to some hyperedge, which must in turn appear in some orbit. These observations give us the necessary information to derive the bound, as they indicate that the number of orbits of $f^{-1}(\boldsymbol{1})$ under $\boldsymbol{T}$ is at least the number of colors in $C$.

\begin{manualtheorem}{1}
    $q(V,E) \leq \frac{1}{|\boldsymbol{T}|} \cdot \sum_{t \in \boldsymbol{T}} |f^{-1}(\boldsymbol{1})^{t}|$.
\end{manualtheorem}

\begin{proof}
    The result is a simple application of Burnside's lemma, which equates the number of orbits of a group acting on a set (denoted by $|f^{-1}(\boldsymbol{1})/\boldsymbol{T}|$) with a particular summation formula $|f^{-1}(\boldsymbol{1})/\boldsymbol{T}| = \frac{1}{|\boldsymbol{T}|} \cdot \sum_{t \in \boldsymbol{T}} |f^{-1}(\boldsymbol{1})^{t}|$, where $|f^{-1}(\boldsymbol{1})^{t}|$ denotes the number of hyperedges in $f^{-1}(\boldsymbol{1})$ that are stabilized by the group element $t \in \boldsymbol{T}$. Since $q(V,E) = |\bar{C}(f^{-1}(\boldsymbol{1}))|$ is less than or equal to the number of orbits of $\boldsymbol{T}$, we arrive at the bound result.
\end{proof}

It is important to note that this bound is not so much computational in its usefulness, since it relies on the construction of a group which in turn relies on precise knowledge of the coloring map and its cardinality, which would already give you $q(V,E)$. Instead, its use is conceptual, in verifying that there are associations between hypergraphs paired with proper minimal coloring maps, and certain subgroups of their automorphism groups, with the number of orbits upper bounding the size of the color set. This idea is more abstract, in that one could imagine a space of these coloring maps, perhaps all defined on some superset of colors with an inclusion ordering, and then a corresponding space of groups perhaps ordered by subgroup inclusion in the overall automorphism group of the hypergraph. It is also interesting to consider the space of all proper coloring maps on a hypergraph, but include non-minimal maps. Then minimality defined by $q(V,E)$ becomes a line drawn through this space, with a corresponding line in the space of groups. There are most likely more nuanced relationships and properties to be gleaned here, using the rigorous group-theoretic connection to aid the hypergraph discoveries.

\vspace{2mm}

The next section is the main bulk of the paper, which lies in the combinatorial work, also centrally relying on the notion of $\bar{C}$ as a coloring map. There is a secondary "induced" coloring map on the vertex set that is derived, which allows for strategic reductions to the complexity of the vertex set with equivalence relations that respect said derived coloring map, and the hyperedge intersection structure. This leads to identifying key terms with which to express the main bounds. This theory is a different branch of the same conceptual tree as the group theory, the base of the tree being the idea of coloring maps. One simply differentiates the two branches by the constructions that are doing the "respecting" of the colorings. On the one hand we have the automorphisms permuting the vertex set that respect $\bar{C}$. And this section to come, we will have equivalence relations that respect the derived coloring map $\bar{C}^{*}$. Both yield upper bounds on the chromatic index.

\section{Combinatorial Upper Bounds on the Chromatic Index}

For the establishment of the following bounds, we rely on a study of color maps, associated equivalence relations, and hypergraphs derived from all of the resultant information, which possess the relevant terms of the conjecture as parts of their hypergraph attributes.

\begin{definition}
    Define the coloring map $\bar{C}^{*}: V \rightarrow C^{*}$, a coloring of the vertex set of our hypergraph which is given by $\bar{C}^{*}(v):=\{\bar{C}(V_{1}), ... , \bar{C}(V_{l})\}$, where $\{V_{1}, ... ,V_{l}\} = [\{v\}, V] \cap f^{-1}(\boldsymbol{1})$.
\end{definition}

In essence, colors of $C^{*}$ are subsets of colors of $C$, where a vertex is mapped to the induced color map $\bar{C}^{*}$ of $\bar{C}$ applied to the collection of all hyperedges containing said vertex, denoted by $[\{v\}, V] \cap f^{-1}(\boldsymbol{1})$-the intersection of the \emph{ultrafilter} of $v$ in $\mathcal{P}(V)$ and the hyperedge set.

\vspace{2mm}

We now can define the first corresponding equivalence relation.

\begin{definition}
    Let $(v_{1}, v_{2}) \in \hspace{1mm}\sim \iff [\{v_{1}\}, V] \cap f^{-1}(\boldsymbol{1}) = [\{v_{2}\}, V] \cap f^{-1}(\boldsymbol{1})$.
\end{definition}

This compresses together all vertices that say belong to a single hyperedge and no other. Hence, it encodes the information about the intersection patterns of the hypergraph into a compressed form.

\begin{proposition}
    $\sim$ is an equivalence relation on $V$.
\end{proposition}

\begin{proof}
    $\sim$ is bi-conditionally identified with the set equality relation, so it inherits reflexivity, symmetry, and transitivity naturally.
\end{proof}

\vspace{2mm}

We can then consider the coloring map $C^{*}$ induced on the quotient structure $V/\sim$, since any two vertices in the same equivalence class by definition have the same set of hyperedges intersecting at both of them, which in this case would have to be a single hyperedge, due to linearity, so possess the same $C^{*}$ values.

\vspace{2mm}

Now, we introduce the crucial definitions that serve as a way to contain our problem.

\vspace{2mm}

We can divide the overall problem into two central cases. 

\begin{enumerate}
    \item The first is when the "flower" of hyperedges $[\{v\}, V] \cap f^{-1}(\boldsymbol{1})$), for a fixed $v$ with $deg_{2}(v) = \Delta([(V,E)]_{2})$, is such that $\bar{C}^{*}([v]_{\sim})$ contains all colors comprising $C$, and in the quantified sense evaluates to all of the colors adding to the value of $q(V,E)$. This would mean that $|\bar{C^{*}}([v]_{\sim})| = \Delta(V,E) = q(V,E)$, as otherwise we would get some vertex with a flower of hyperedges containing it, possessing the max degree of the hypergraph (so with larger hypergraph degree than $v$), meaning there would be additional hyperedges needing additional colors not in $\bar{C^{*}}([v]_{\sim})$. This case immediately entails that $q(V,E) = \Delta(V,E) < \Delta([(V,E)]_{2}) + 1$, so the conjecture holds trivially.

    \item The second case is when $\exists c_{0} \in C$ such that $c_{0} \notin \bar{C^{*}}([v]_{\sim})$, which entails a much more complex mathematical situation where the problem remains open. This is the logical case we focus on for the rest of the paper, where the assumption is that $\exists c_{0} \notin \bar{C^{*}}([v]_{\sim})$. All main results of this material fall within this case logical case.
\end{enumerate}

\begin{definition}
    Define $\Omega:=\{\forall [u]_{\sim} \in V/\sim|: \exists c_{i}, V_{i}, \hspace{1mm} \bar{C}(V_{i}) = c_{i} \in \bar{C^{*}}([u]_{\sim}) \land c_{i} \notin \bar{C^{*}}([v]_{\sim})\}$.
\end{definition}

\begin{definition}
    Let $([v_{1}]_{\sim}, [v_{2}]_{\sim}) \in \theta \iff \bar{C^{*}}([v_{1}]_{\sim}) = \bar{C}^{*}([v_{2}]_{\sim})$, where $\theta$ is defined specifically on the set $\Omega$.
\end{definition}

\begin{proposition}
    $\theta$ is an equivalence relation on $\Omega$.
\end{proposition}

\begin{proof}
    Since $\theta$ is defined bi-conditionally in terms of set equality, it inherits the reflexivity, symmetry, and transitivity.
\end{proof}

This equivalence relation serves to shave off any redundancy in the subset color set of $C^{*}$. For example, we could have sufficient symmetry in the hypergraph and a coloring such that two distinct flowers of hyperedges are colored the same, and in a relative sense there is redundancy in the usage of colors, or at least not a need for additional colors in the two regions. Also, $\bar{C^{*}}$ can be induced as a coloring map on $\Omega/\theta$ as well, since we still have that every pair of elements in a given equivalence class of $\Omega/\theta$ will have the same coloring under $\bar{C^{*}}$.

We will now identify specific subsets of of the quotient set $\Omega/\theta$ each with respect to a fixed color not in $\bar{C^{*}}([v]_{\sim})$. Mainly, since we assume for this case that $\exists c_{0} \in C$ such that $c_{0} \notin \bar{C^{*}}([v]_{\sim})$, fix this color $c_{0}$. If there is only this one additional color, then $q(V,E) = |\bar{C^{*}}([v]_{\sim})| + 1 \leq \Delta(V,E) +1$, which still entails that the conjecture holds here. So, we will assume that there are potentially many more colors apart from just $c_{0}$ that are not in $\bar{C^{*}}([v]_{\sim})$.

\begin{proposition}
    Consider $c_{i} \in C$ such that $c_{i} \neq c_{0}$ and $c_{i} \notin \bar{C^{*}}([v]_{\sim})$. Then there exists a $u_{i} \in V$ such that $[[u_{i}]_{\sim}]_{\theta} \in \Omega/\theta$ and $\{c_{0}, c_{i}\} \subseteq \bar{C^{*}}([[u_{i}]_{\sim}]_{\theta})$. 
\end{proposition}

\begin{proof}
    Firstly, suppose for a contradiction that every hyperedge $V_{i}$ with $\bar{C}(V_{i}) = c_{i}$ does not intersect with any hyperedges $V_{j}^{0}$ with $\bar{C}(V_{j}^{0}) = c_{0}$. Then due to the minimality of $\bar{C}$ we could simply color all hyperedge representatives of $c_{i}$ with $c_{0}$ instead, and have no proper coloring violations while reducing the number of colors needed. So, we have via contradiction that $\exists V_{i}$ colored with $c_{i}$ and that intersects with a hyperedge $V_{j}^{0}$ with $\bar{C}(V_{j}^{0}) = c_{0}$. This means that $\exists u_{i}$ with $u_{i} \in V_{i} \cap V_{j}^{0}$ meaning $\{c_{0}, c_{i}\} \subseteq \bar{C^{*}}([[u_{i}]_{\sim}]_{\theta})$. The same reasoning can be applied to all colors $c_{k} \neq c_{0}$ and with $c_{k} \notin \bar{C^{*}}([v]_{\sim})$.
\end{proof}

With this proposition in hand, we can refine to a set smaller than $\Omega/\theta$.

\begin{definition}
    For our fixed color $c_{0}$, define $\Gamma_{c_{0}}:=\{\forall [[u_{i}]_{\sim}]_{\theta} \in \Omega/\theta|: \exists c_{i} \notin \bar{C^{*}}([v]_{\sim}) \land c_{i} \neq c_{0} \land \{c_{0}, c_{i}\} \subseteq \bar{C^{*}}([[u_{i}]_{\sim}]_{\theta})\}$.
 \end{definition}

There is a naturally defined hypergraph, most typically non-linear and non-uniform depending on the underlying $(V,E)$, which via its typical hypergraph properties expresses and relates $q(V,E)$ and $\Delta([(V,E]_{2}) + 1$, as well as other key quantities surrounding the problem.

\begin{definition}
    Define the hypergraph $H^{*}_{\Gamma_{c_{0}}}:= (\cup_{[[u_{i}]_{\sim}]_{\theta} \in \Gamma_{c_{0}}} \bar{C^{*}}([[u_{i}]_{\sim}]_{\theta}), \bar{C^{*}}(\Gamma_{c_{0}}))$.
\end{definition}

$\bar{C^{*}}(\Gamma_{c_{0}})$ denotes the point-wise application of $\bar{C^{*}}$ on the collection of elements of $\Gamma$. This hypergraph has colors of $C$ as its vertices, and color subsets for hyperedges. Note first that the hypergraph has no duplicate hyperedges, since we are dealing with a subset of $\Omega/\theta$, which in itself has all distinct color subsets assigned to its elements by definition. Moreover, it is loopless, in that every $|\bar{C^{*}}([[u_{i}]_{\sim}]_{\theta})| > 1$, containing $c_{0}$ and some other $c_{i}$. It not likely to be linear or uniform, but those are not particularly prudent conditions to hold over such a structure. It has a hypergraph maximum degree $\Delta(H^{*}_{\Gamma_{c_{0}}}) = |\Gamma_{c_{0}}|$, corresponding to the point of intersection $c_{0}$, which lies at the intersection of all the hyperedges. It has a hypergraph \emph{rank} (the size of the largest hyperedge) which is equal to or less than $\Delta(V,E)$, which upper bounds all the color subsets assigned by $\bar{C^{*}}$. There are three propositions that can help connect this hypergraph to $q(V,E)$ very directly.

Before the following results, there are two connected inequalities that will be used concerning $\Delta([(V,E)]_{2})$.

\begin{lemma}
    For a linear loopless hypergraph $(V,E)$, we have:
    
    $(ar(V,E) - 1) \cdot \Delta(V,E) \leq \Delta([(V,E)]_{2}) \leq (r(V,E) - 1) \cdot \Delta(V,E)$.
\end{lemma}

\begin{proof}
    Suppose $(ar(V,E) - 1) \cdot \Delta(V,E) >\Delta([(V,E)]_{2})$. Then this would imply any vertex with max degree $\Delta(V,E)$ in the hypergraph would have greater two-section degree than $\Delta([(V,E)]_{2})$, since $deg_{2}(u):=\sum_{V_{i} \in E} (|V_{i}| - 1)$ and every hyperedge has to be at least the antirank-hence a contradiction. If we suppose $\Delta([(V,E)]_{2}) > (r(V,E) - 1) \cdot \Delta(V,E)$, then this would imply that for $v \in V$ with $deg_{2}(v) = \Delta([(V,E)]_{2})$ has more vertices in the flower of hyperedges containing it than a vertex $u$ with max degree $\Delta(V,E)$ and with all hyperedges size $r(V,E)$, indicating that the hypergraph degree of $v$ is greater than that of $u$-a contradiction as well.
\end{proof}

For the following propositions, it is interesting to point out that in terms of serving as upper bounds to $q(V,E)$, the right-hand side of the first inequality is most tight as an upper bound, and each gets less so, in sequential order. So through the goal of bringing the $\Delta([(V,E)]_{2})$ term into the right hand side of the inequalities, we must make some admissions in terms of bound tightness.

\begin{proposition}
    $q(V,E) \leq |\bar{C^{*}}([v]_{\sim})| + |\cup_{[[u_{i}]_{\sim}]_{\theta} \in \Gamma} \bar{C^{*}}([[u_{i}]_{\sim}]_{\theta})|$.
\end{proposition}

\begin{proof}
    This formula simply involves comparing the number of elements in $C$ with the right hand side of the inequality, which counts all colors in $\bar{C^{*}}([v]_{\sim})$, plus all colors not in $\bar{C^{*}}([v]_{\sim})$ (but with some potential repeats of colors in $\bar{C^{*}}([v]_{\sim})$-hence the inequality).
\end{proof}

\begin{proposition}
    $q(V,E) \leq |\bar{C^{*}}([v]_{\sim})| + \Delta([H^{*}_{\Gamma_{c_{0}}}]_{2}) + 1$.
\end{proposition}

\begin{proof}
    Since we essentially have a star-shaped/flower hypergraph in $H^{*}_{\Gamma_{c_{0}}}$ where all hyperedges intersect with the fixed $c_{0}$, the two-section degree of $c_{0}$ as a vertex in $H^{*}_{\Gamma_{c_{0}}}$ is equal to $\Delta([H^{*}_{\Gamma_{c_{0}}}]_{2})$, and since clearly the two-section degree of $c_{0}$ plus $1$ counts precisely all colors of $\cup_{[[u_{i}]_{\sim}]_{\theta} \in \Gamma_{c_{0}}} \bar{C^{*}}([[u_{i}]_{\sim}]_{\theta})$, then the inequality follows suit.
\end{proof}
    
\begin{proposition}
    $q(V,E) \leq |\bar{C}^{*}([v]_{\sim})| + (\Delta(V,E) - 1) \cdot |\Gamma_{c_{0}}| + 1$.
\end{proposition}

\begin{proof}
    This follows from the prior established facts about the two-section degree of a hypergraph, which in our case is that $\Delta([H^{*}_{\Gamma_{c_{0}}}]_{2}) + 1 \leq (r(H^{*}_{\Gamma_{c_{0}}}) - 1) \cdot \Delta(H^{*}_{\Gamma_{c_{0}}}) + 1$, so one then just has to substitute in the facts that $r(H^{*}_{\Gamma_{c_{0}}}) \leq \Delta(V,E)$ and $\Delta(H^{*}_{\Gamma_{c_{0}}}) = |\Gamma_{c_{0}}|$.
\end{proof}

\begin{manualtheorem}{2}
    $q(V,E) + |\Gamma_{c_{0}}|\leq \frac{|\Gamma_{c_{0}}|+1}{ar(V,E) - 1}\cdot \Delta([(V,E)]_{2}) + 1$.
\end{manualtheorem}

\begin{proof}
    $\rightarrow q(V,E) \leq |C^{*}([v]_{\sim})| + (\Delta(V,E) - 1) \cdot |\Gamma_{c_{0}}| + 1$, 

    \vspace{3mm}

    $\rightarrow q(V,E) + |\Gamma_{c_{0}}| \leq |C^{*}([v]_{\sim})| + |\Gamma_{c_{0}}| \cdot \Delta(V,E) +1$,

    \vspace{3mm}

    $\rightarrow q(V,E) + |\Gamma_{c_{0}}| \leq \Delta(V,E) + |\Gamma_{c_{0}}| \cdot \Delta(V,E) + 1$, 

    \vspace{3mm}

    $\rightarrow (ar(V,E) - 1) \cdot (q(V,E) + |\Gamma_{c_{0}}|) \leq (ar(V,E) - 1)\cdot \Delta(V,E) \cdot (|\Gamma_{c_{0}}|+1) + (ar(V,E) - 1)$, 

    \vspace{1mm}
    
    and since $(ar(V,E) - 1)\cdot \Delta(V,E) \leq \Delta([(V,E)]_{2})$, we get:

    \vspace{3mm}

    $\rightarrow q(V,E) + |\Gamma_{c_{0}}|\leq \frac{|\Gamma_{c_{0}}|+1}{ar(V,E) - 1}\cdot \Delta([(V,E)]_{2}) + 1$.
\end{proof}

\begin{corollary}
    If there exists $c_{0} \notin \bar{C}^{*}([v]_{\sim})$ with $|\Gamma_{c_{0}}| + 1 \leq ar(V,E) - 1$, then $q(V,E) \leq \Delta([(V,E)]_{2}) + 1$.
\end{corollary}

\begin{proof}
    This follows from the bound in Theorem 2, as an immediate consequence when there exists $\Gamma_{c_{0}}$ such that $|\Gamma_{c_{0}}| + 1 \leq ar(V,E) - 1$.
\end{proof}

\vspace{2mm}

For the other main theorem of this section, we can derive a hypergraph from $(V,E)$, akin to taking all $\Gamma_{c_{i}}$ sets for $c_{i} \notin \bar{C}^{*}([v]_{\sim})$ as the hyperedges (but just considering the actual equivalence classes of the form $[[u_{i}]_{\sim}]_{\theta}$ rather than their images under $\bar{C}^{*}$ as vertices). It is then shown that the Berge-F\"{u}redi conjecture holds for a $(V,E)$ and $\bar{C}$ when this derived hypergraph has the \emph{Helly property}.

\begin{definition}
    A hypergraph is said to have the \emph{Helly property} when for every $n$ hyperedges $V_{1}, ... ,V_{n}$, if for every $V_{i}, V_{j}$ with $i, j \in [n]$ we have $V_{i} \cap V_{j} \neq \emptyset$, then the property implies $V_{1} \cap ...\cap V_{n} \neq \emptyset$.
\end{definition}

\begin{definition}
    Let $H_{\Gamma}:=(\cup_{c_{i} \in C/\bar{C}^{*}([v]_{\sim})} \Gamma_{c_{i}}, \{\Gamma_{c_{i}}|: c_{i} \in C/\bar{C}^{*}([v]_{\sim})\})$ be a hypergraph whose vertices are specific elements of $\Omega/\theta$, and whose hyperedges are the $\Gamma$ sets.
\end{definition}

\begin{manualtheorem}{2.1}
    For $(V,E)$, $\bar{C}$ with $ar(V,E) \geq 3$, suppose that they possess a derived hypergraph $H_{\Gamma}$ that has the Helly property. Then $q(V,E) \leq \Delta([(V,E)]_{2}) + 1$.
\end{manualtheorem}

\begin{proof}
    We already know that for any $c_{i}, c_{j} \notin C/\bar{C}^{*}([v]_{\sim})$, $\exists [[u_{ij}]_{\sim}]_{\theta} \in \Gamma_{c_{i}} \cap \Gamma_{c_{j}}$, for $\{c_{i}, c_{j}\} \subseteq \bar{C}^{*}([[u_{ij}]_{\sim}]_{\theta})$. So, if we assume the Helly property applies to $H_{\Gamma}$, then we know $\exists [[u]_{\sim}]_{\theta} \in \Omega/\theta$ such that $[[u]_{\sim}]_{\theta} \in \Gamma_{c_{1}} \cap ... \cap \Gamma_{c_{l}}$ which ensures its non-emptiness, with $l = |C/\bar{C}^{*}([v]_{\sim})|$. By definition, $C/\bar{C}^{*}([v]_{\sim}) \subseteq \bar{C}^{*}([[u]_{\sim}]_{\theta})$, hence $|C/\bar{C}^{*}([v]_{\sim})| \leq |\bar{C}^{*}([[u]_{\sim}]_{\theta})| \leq \Delta(V,E)$. So via $q(V,E) = |\bar{C}^{*}([v]_{\sim})| + |C/\bar{C}^{*}([v]_{\sim})|$ we get that $q(V,E) \leq 2 \cdot \Delta(V,E)$, so the conjecture holds for $ar(V,E) \geq 3$. 
\end{proof}

Intuitively, an $H_{\Gamma}$ would possess the Helly property when the hyperedges of $(V,E)$ representing all colors not in $\bar{C}^{*}([v]_{\sim})$ all intersect at some vertex-this vertex defining the common equivalence class which allows for the Helly property to be satisfied, as it belongs to every hyperedge $\Gamma_{ci}$, reflecting that it must contain all colors not in $\bar{C}^{*}([v]_{\sim})$.

\vspace{2mm}

Lastly, we have another condition for the conjecture holding, which invokes the Helly-type result above. 

\begin{lemma}
    Let $(V,E)$ be a $k$-uniform, linear hypergraph. If a family of hyperedges $\mathcal{F}\subseteq E$ forms a clique (is pair-wise intersecting), such that $|\mathcal{F}|> k^{2}-k+1$, then all hyperedges of $\mathcal{F}$ intersect at a single vertex.
\end{lemma}

This result is made standard in \cite{NaikRaoShrikhandeSinghi1980}, and will help us specifically when all the hyperedges of such a clique possess precisely every color not in $\bar{C^{*}}([v]_{\sim})$ when in this particular logical case.

\begin{manualcor}{2.2}
    Let $(V,E)$ be a $k$-uniform, linear hypergraph for $k \geq 3$, with a proper minimal hyperedge coloring $\bar{C}$. Suppose further that, for fixed $v \in V$ with $deg_{2}(v) = \Delta([(V,E)]_{2})$, there exists a clique of hyperedges (all pairwise intersecting) $\mathcal{F} \subseteq E$ with the nonempty set of colors $C/\bar{C}^{*}([v]_{\sim})$, such that $\bar{C}(\mathcal{F}) = C/\bar{C}^{*}([v]_{\sim})$.Then,  $|\mathcal{F}| > k^{2} -k + 1$ implies that $q(V,E) \leq \Delta([(V,E)]_{2}) + 1$.
\end{manualcor}

\begin{proof}
    The assumptions are that we have $\mathcal{F}$ with $\bar{C}(\mathcal{F}) = C/\bar{C}^{*}([v]_{\sim})$, and since $|\mathcal{F}| > k^{2}-k+1$, via the above lemma we get that they all intersect at a single vertex $u \in V$. Hence, $C/\bar{C}^{*}([v]_{\sim}) \subseteq \bar{C}^{*}([[u]_{\sim}]_{\theta})$. So for a given $c_{i} \in C/\bar{C}^{*}([v]_{\sim})$, this reasoning yields the $q(V,E) \leq 2 \cdot \Delta(V,E)$ argument from before, and so the conjecture follows.
\end{proof}

Although fairly specific in nature, the bounds of this paper go beyond numerical-based thresholds to guarantee the conjecture holds in a particular case. Our approach is structural, and replaces the parameter assumptions with a verifiable intersection pattern, capturing instances that fall outside known thresholds.

\section{Examples}

We first compute a simple example of a colored hypergraph with a nontrivial $\boldsymbol{T}$ group and the bounds it yields. In Figure \ref{fig1}, the $\boldsymbol{T}$ group can be ascertained by observing the five distinct transpositions that can swap pairs of vertices within hyperedges, along with the larger map which transposes the two halves of the hypergraph. The size of the group is $64$, which comes from the generators $\rho_{1}, \rho_{1}', \rho_{2}, \rho_{2}', \rho_{3}$, and $\delta$, such that $\delta \circ \rho_{1} \circ \delta^{-1} = \rho_{1}'$, $\delta \circ \rho_{2} \circ \delta^{-1} = \rho_{2}'$, and $\delta \circ \rho_{2} = \delta$. This allows us to determine the partitioning of the hyperedges as fixed by the group elements in our sum. Mainly, the bound gives $4 \leq \frac{1}{64} \cdot \sum_{t \in T}|f^{-1}(\boldsymbol{1})|^{t}$, where the only elements that don't stabilized the entire hyperedge set are those involving $\delta$. So this gives $32$ elements which involve $\delta$ in the composition, each one of which stabilizes just the single "bridge" hyperedge connecting the two halves. Hence, $4 \leq \frac{(32 + 32 \cdot 11)}{64} = 6$, showing that we can yield nice upper bounds with this when the coloring map and symmetries align.

\begin{figure}[!h]
    \centering
    \includegraphics[width=0.7\linewidth]{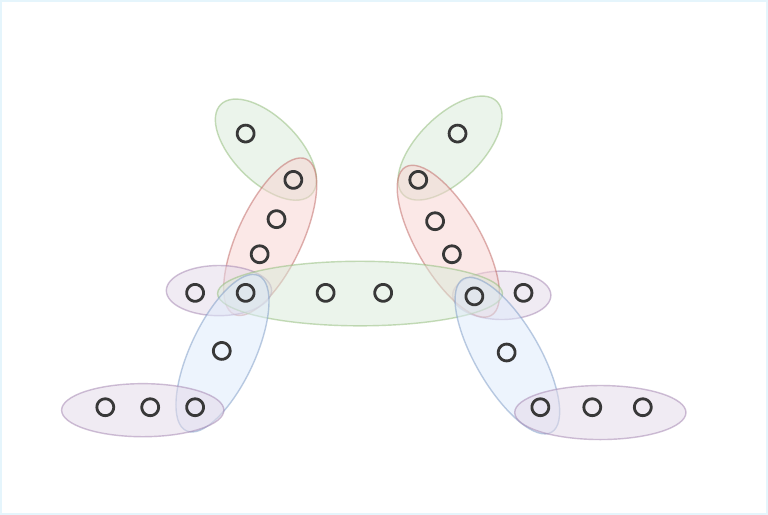}
    \caption{Group-theoretic bound example}
    \label{fig1}
\end{figure}

\begin{figure}
    \centering
    \includegraphics[width=.75\linewidth]{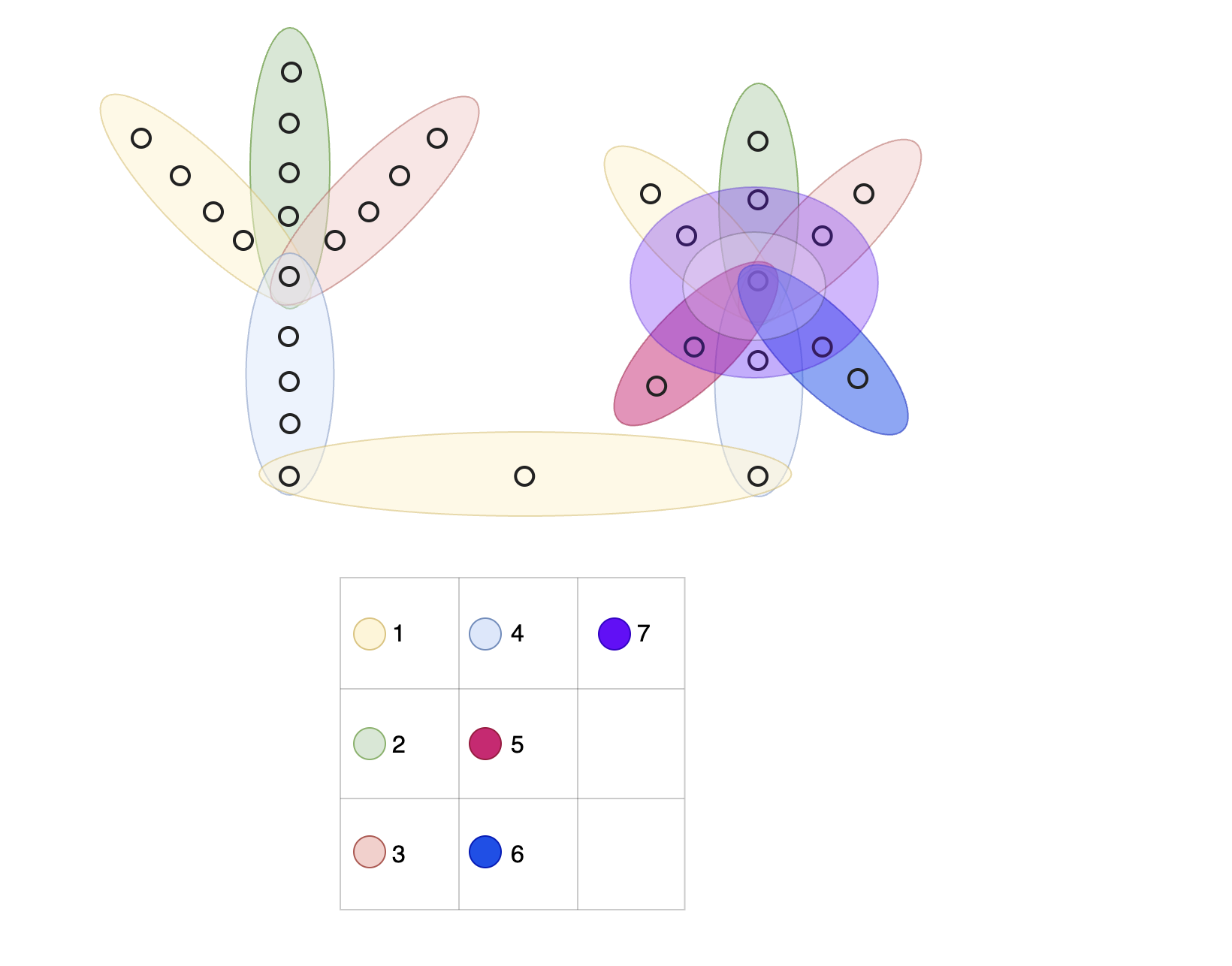}
    \caption{A linear hypergraph $(V,E), \bar{C}$ with its proper minimal coloring of 7 colors.}
    \label{fig2}
\end{figure}

\begin{figure}
    \centering
    \includegraphics[width=0.75\linewidth]{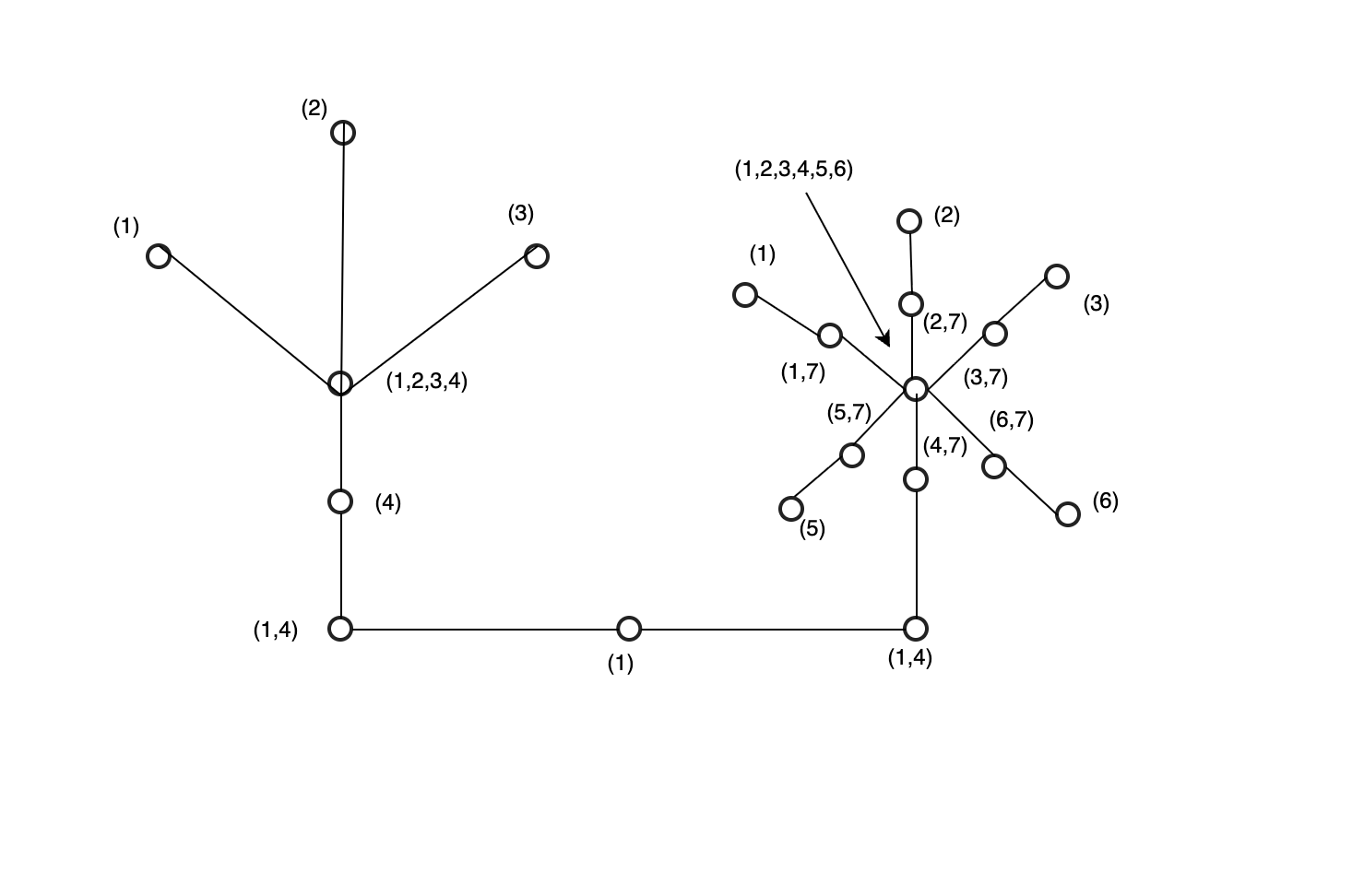}
    \caption{This graph represents $V/\sim$ with its equivalence classes as vertices.}
    \label{fig3}
\end{figure}

\begin{figure}
    \centering
    \includegraphics[width=0.75\linewidth]{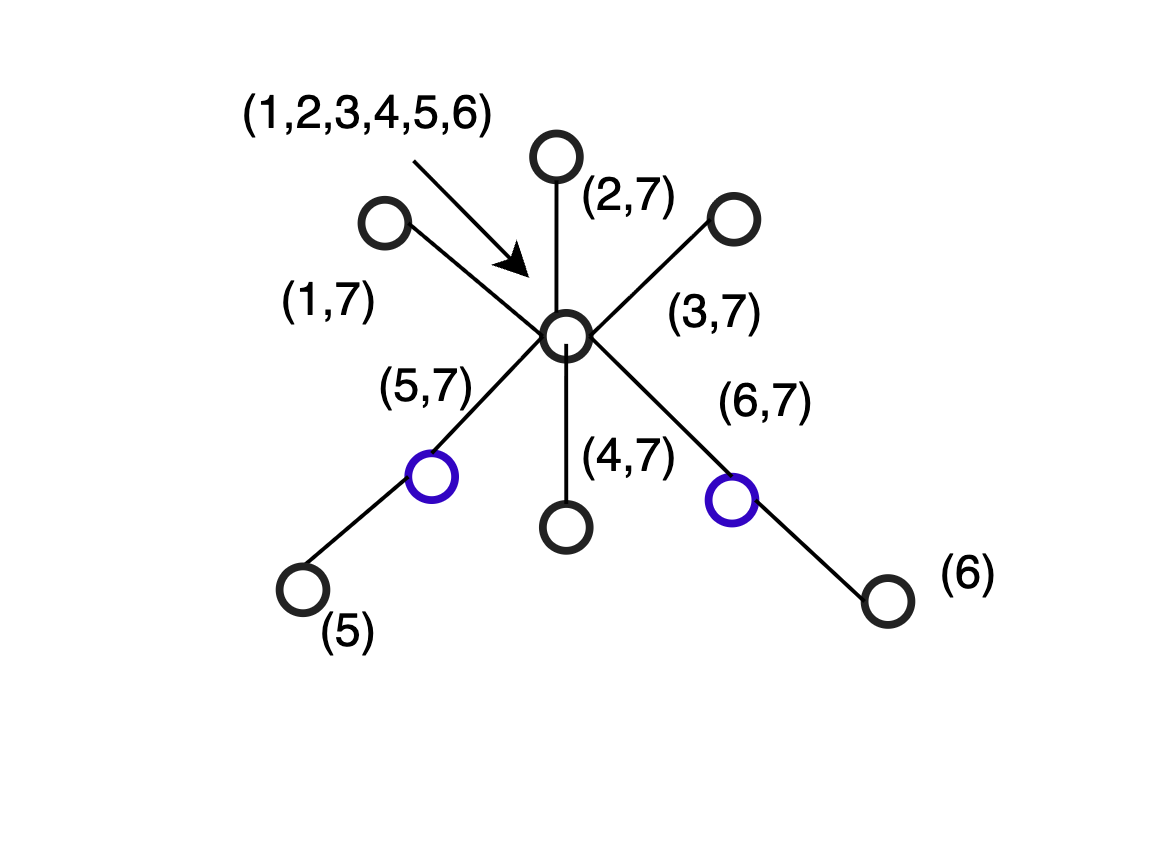}
    \caption{This subgraph shows $\Omega/\theta$, with the two vertices defining $\Gamma_{7}$ highlighted in dark blue.}
    \label{fig4}
\end{figure}

\begin{figure}
    \centering
    \includegraphics[width=0.75\linewidth]{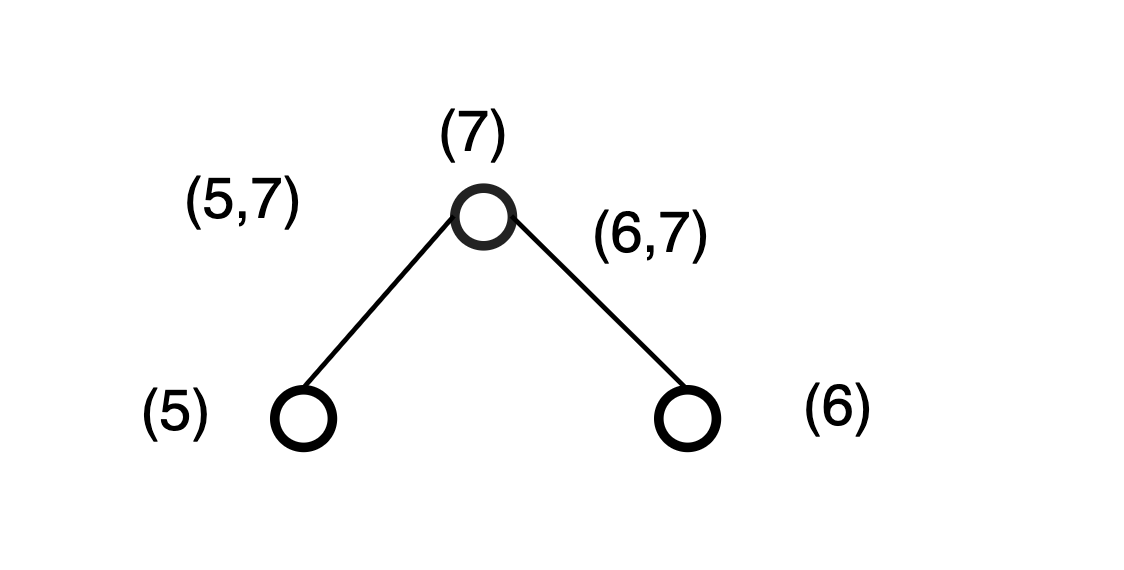}
    \caption{This is the hypergraph (just a graph in this case) $H_{\Gamma_{7}}$, whose vertices are colors}
    \label{fig5}
\end{figure}

For Figures \ref{fig2}, \ref{fig3}, \ref{fig4}, \ref{fig5}, we can compute the bounds given by the combinatorial theory. Figure \ref{fig2} shows the hypergraph with its coloring. Figure \ref{fig3} transitions to visualizing $V/\sim$, where one can see the collapse of the "vacuous" vertices contained within singular hyperedges. Figure \ref{fig4} shows $\Omega/\theta$, with the two elements of $\Gamma_{7}$ colored in, and Figure \ref{fig5} shows the resulting derived (hyper)graph $H_{\Gamma_{7}}$. $q(V,E) + |\Gamma_{7}| \leq \frac{|\Gamma_{7}| + 1}{ar(V,E) - 1} \cdot \Delta([(V,E)]_{2}) + 1$ holds (in this case for the fixed color $7$), since from the examples we get $7 + 2 \leq \frac{2}{2} \cdot 20 + 1$, so $9 \leq 21$. Note that the graph edge sets of the graphs for $V/\sim$ and $\Omega/\theta$ are imposed onto the actual set collections, to show the transition from one structure down through the modded version, all being underlying form of the hypergraph. Once we "mod" with $\sim$, it is more apt to think in terms of subgraphs of the two-section graph as opposed to sub-hypergraphs as the edge-components to our quotient objects. In general, the sequence of the figures shows starting with the initial $(V,E)$ with its proper minimal coloring $\bar{C}$ applied to the hyperedges, then moving to each proceeding structure once modded by the two respective relations $\sim$ and $\theta$. Then the relevant $\Gamma$ set is isolated and its nontrivial color hypergraph $H^{*}_{\Gamma_{7}}$ is shown.

\vspace{2mm}

Figures \ref{fig6}, \ref{fig7}, \ref{fig8}, \ref{fig9}, \ref{fig10} show the constructions for a hypergraph whose $H_{\Gamma}$ has the Helly property. This follows from the construction $(V,E)$ with its $\bar{C}$ possessing a vertex which is the intersection point of hyperedges representing all colors of $C/\bar{C}^{*}([v]_{\sim})$. Hence, the double equivalence class of said vertex is the common intersection point for all $\Gamma_{ci}$, and where pairwise intersection is achieved by proposition 4.7. Clearly, $q(V,E) \leq 2 \cdot \Delta(V,E)$ and $q(V,E) \leq \Delta([(V,E)]_{2}) + 1$.

\begin{figure}
    \centering
    \includegraphics[width=0.75\linewidth]{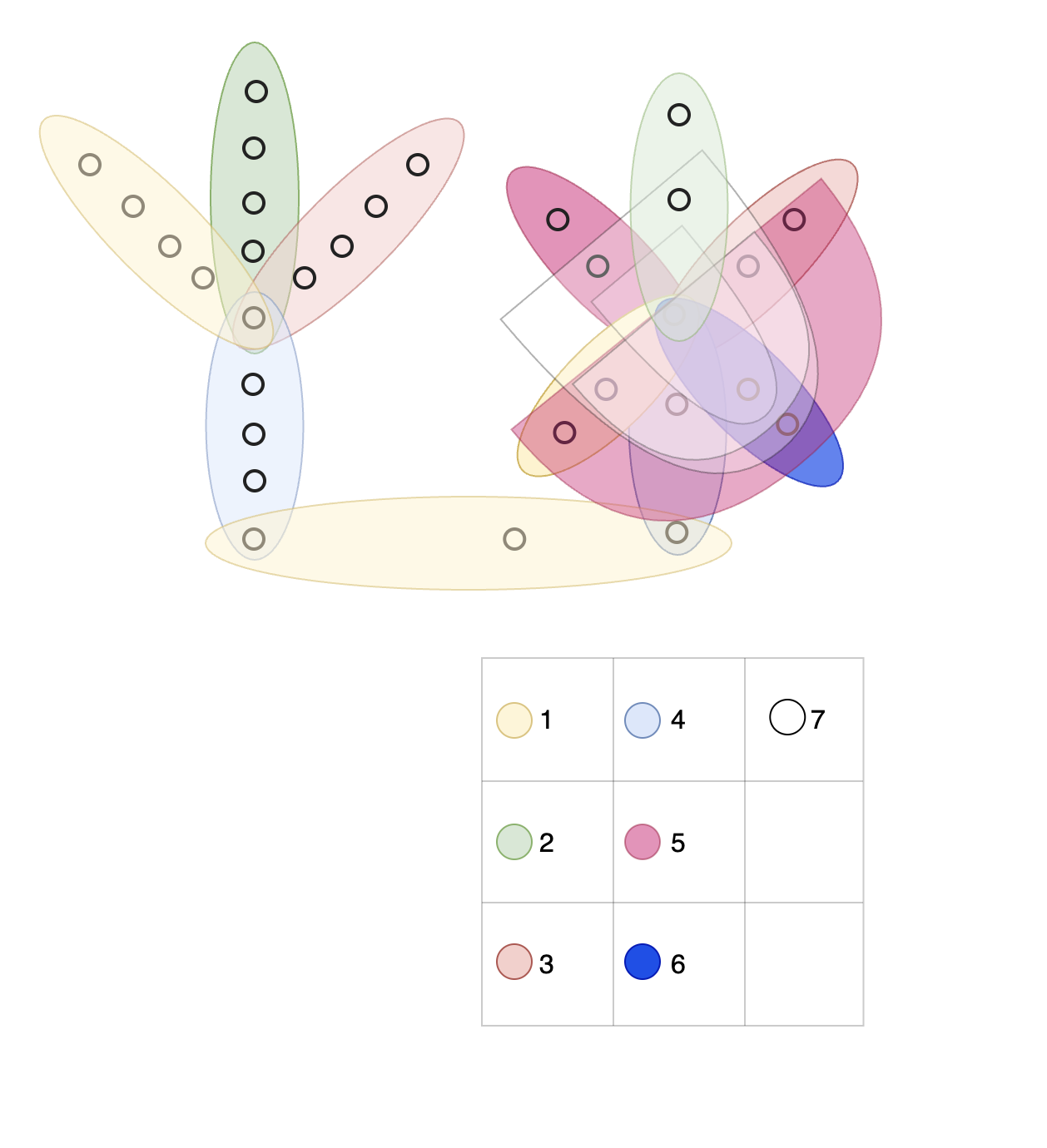}
    \caption{This figure shows a hypergraph whose $H_{\Gamma}$ exhibits the Helly property.}
    \label{fig6}
\end{figure}

\vspace{4cm}

\begin{figure}
    \centering
    \includegraphics[width=0.75\linewidth]{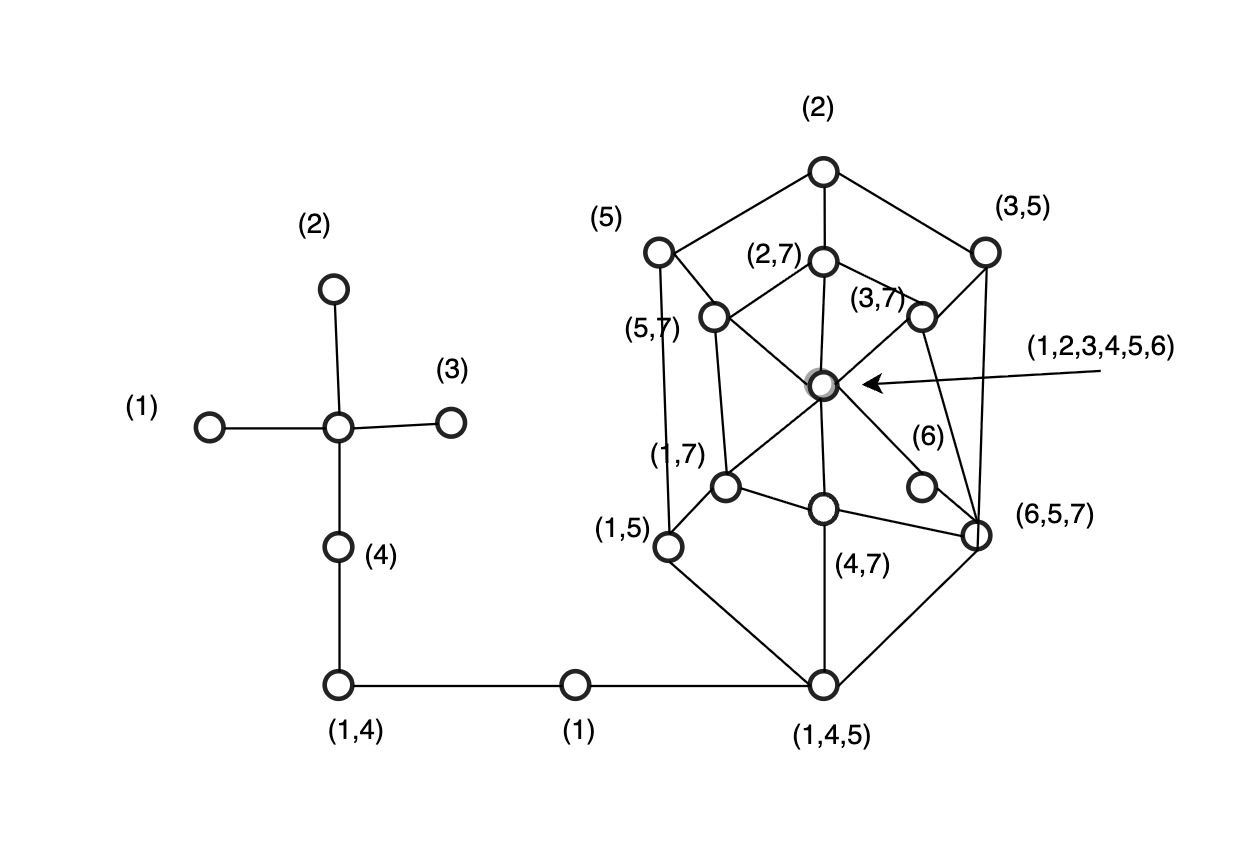}
    \caption{$V/\sim$}
    \label{fig7}
\end{figure}

\begin{figure}
    \centering
    \includegraphics[width=0.70\linewidth]{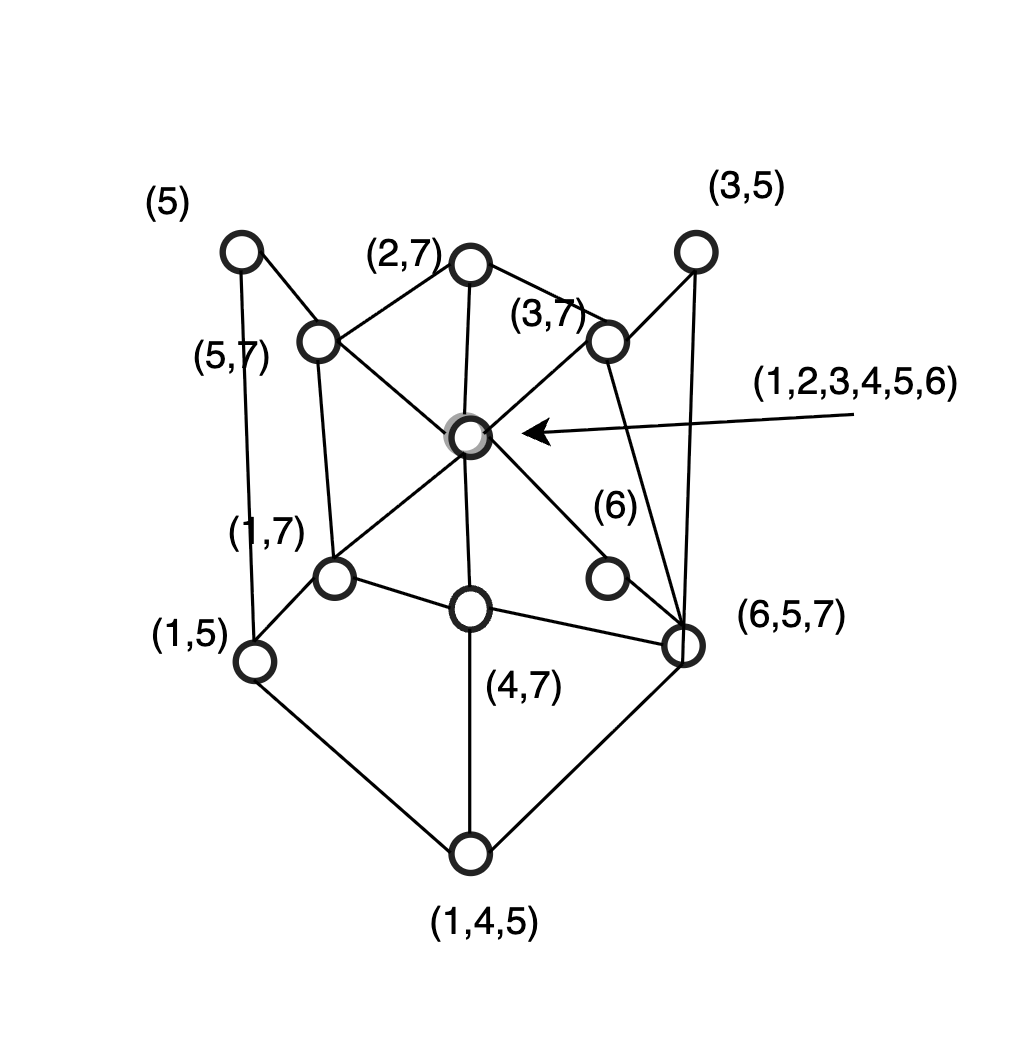}
    \caption{$\Omega/\theta$}
    \label{fig8}
\end{figure}

\begin{figure}
    \centering
    \includegraphics[width=0.75\linewidth]{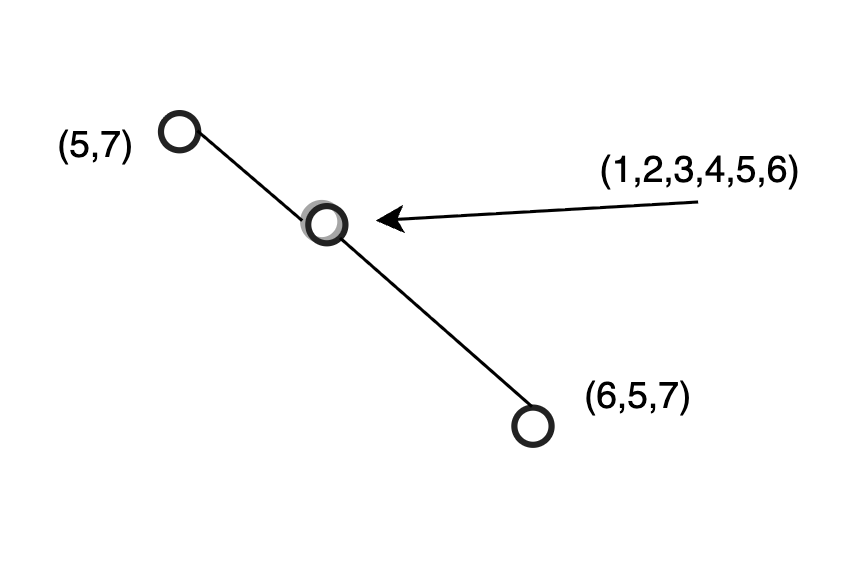}
    \caption{This depicts the vertices of $H_{\Gamma}$, from which clearly all $\Gamma_{c_{i}}$ contain the equivalence class with $\bar{C}^{*}$ color subset comprised of $6,5,7$}
    \label{fig9}
\end{figure}

\begin{figure}
    \centering
    \includegraphics[width=0.75\linewidth]{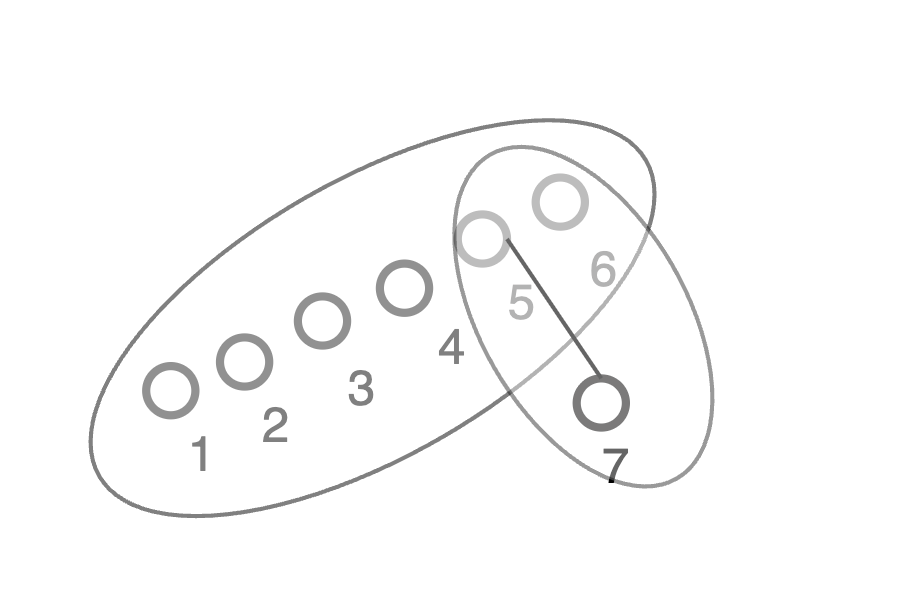}
    \caption{The $H^{*}_{\Gamma_{5}}$ hypergraph}.
    \label{fig10}
\end{figure}

\clearpage

\section{Conclusions and Future Directions}

To conclude, we have demonstrated two different types of upper bounds on the chromatic index of linear loopless hypergraphs. The first is from Theorem 1, which is more deductive in nature, following as a consequence of the existence of potentially non-trivial subgroups of the hypergraph automorphism group tied to minimal proper coloring maps of such hypergraphs. The second more constructive bound of the paper is from Theorem 2, which can be defined for any $c_{0} \in C/\bar{C}^{*}([v]_{\sim})$, and if a $\Gamma_{c_{0}}$ is of size equal to or less than the antirank minus one, then the Berge-F\"{u}redi conjecture holds. On this note, we establish two other distinct but related sufficient conditions for when the conjecture in question must hold. The first occurs when the derived hypergraph $H_{\Gamma}$ has the Helly property, turning its pairwise-intersecting property into full intersection at a existing element, implying the conjecture must hold through the logic of the surrounding framework. The second related condition is specified to apply to $k$-uniform linear hypergraphs, and leverages a result of \cite{NaikRaoShrikhandeSinghi1980}. It says that, if a subcollection of hyperedges which serve as representatives for all colors of $C/\bar{C}^{*}([v]_{\sim})$ is pairwise intersecting, as well as of size greater than $k^{2}-k+1$, then the conjecture holds.

\vspace{2mm}

There are three main avenues for future research directions. The first is at the level of the background theory, pertaining to the correspondence between satisfiable Boolean functions and hypergraphs, via the isomorphism between the power set and the Hamming space. This can let one translate problems back and forth between the two domains and draw computational parallels between computability problems of hypergraphs and those of Boolean functions. For instance, the class of $n$-place Boolean functions constrained to the condition of only (potentially) evaluating strings as true if they have precisely $k$ "$\boldsymbol{1}$" entries for a fixed $k$ with $1 \leq k \leq n$ , would be the same as restricting the class of hypergraphs (of $n$ vertices) to just $k$-uniform hypergraphs. The theory of hypergraphs and its logically stronger counterparts in matroids, abstract simplicial complexes, etc., could be used to provide an analogous scaffolding for Boolean functions and their classification.

\vspace{2mm}

The second direction is at the level of the group theory. This paper has just presented one main bound concerning the subgroup $\boldsymbol{T}$ of the full hypergraph automorphism group, in relation to the coloring map $\bar{C}$ and how it respects these symmetries. However, more nuanced group-theoretic information might be teased out of these hypergraph colorings, in the form of properties about their associated coloring groups that goes beyond averaging the size and number of orbits. There is also a general approach one could take, in studying the space of all proper (possibly minimal as well) coloring functions on a given hypergraph, and the corresponding space of color-preserving subgroups of the automorphism group. The general idea is that the groups which emerge from a hypergraph equipped with a coloring map have encoded in them information about said structures, such as the number of orbits upper bounding the cardinality of the color set.

\vspace{2mm}

The third and perhaps most pertinent direction of research is with the combinatorial side of the paper, in exploring further analysis of the derived hypergraphs involved, as well as the related $\Gamma$ sets. Other sufficient conditions for the conjecture to hold may be found using these tools, and studying the intrinsic hypergraph properties of the $H^{*}_{\Gamma_{c_{i}}}$ and $H_{\Gamma}$ hypergraphs. In general, there are two sides to the hypergraphs derived from the initial $(V,E)$. On one side are the hypergraphs like $H_{\Gamma}$ which consist of the actual equivalence classes or sets of equivalence classes, whereas on the other side we have $H_{\Gamma_{c_{i}}}^{*}$, defined as consisting of color subsets. So fully reasoned out, there is the collection of $H_{\Gamma_{ci}}$ and their limiting object $H_{\Gamma}$ with all the $\Gamma$ sets (of equivalence classes) as hyperedges, and then all the $H_{\Gamma_{c_{i}}}^{*}$ and their limiting object $H_{\Gamma}^{*}$, comprised of subsets of colors, with $\bar{C}^{*}$ the connecting map.

Another consideration is to think of $\Gamma$ as a mapping from the domain $C/\bar{C}^{*}([v]_{\sim})$ to the value assigned with $\frac{|\Gamma_{c_{0}}|+1}{ar(V,E) - 1}\cdot \Delta([(V,E)]_{2}) + 1$, for each $c_{0} \in C/\bar{C}^{*}([v]_{\sim})$. The question then becomes one relating to minimizing $\Gamma$ as a function, enough to be $\leq \Delta([(V,E)]_{2}) +1$.

\section*{Acknowledgments}

We would like to cordially thank the organizers of the Wolfram Summer School 2024, where part of this work was initiated.

\end{document}